\documentclass[a4paper,12pt]{amsart}

\usepackage{amssymb}
\usepackage[utf8]{inputenc}
\usepackage[T1]{fontenc}
\usepackage{enumitem}
\setenumerate[1]{label={\bf[\alph{*}]}, ref={\bf[\alph{*}]}}

\newcommand{\ra}{\rightarrow}
\newcommand\enn{\mathbb N}

\newcommand{\spann}{\mbox{span}}
\newcommand{\cspann}{\overline{\mbox{span}}}

\newcommand{\HB}{\text{H{\kern -0.35em}B}}
\newcommand{\vertiii}[1]{{\left\vert\kern-0.25ex\left\vert\kern-0.25ex\left\vert#1\right\vert\kern-0.25ex\right\vert\kern-0.25ex\right\vert}}
\newcommand{\overbar}[1]{\mkern 1.5mu\overline{\mkern-1.5mu#1\mkern-1.5mu}\mkern 1.5mu}

\newtheorem{thm}{Theorem}[section]
\newtheorem{prop}[thm]{Proposition}
\newtheorem{lem}[thm]{Lemma}

\theoremstyle{definition}
\newtheorem{defn}[thm]{Definition}

\theoremstyle{remark}

\newtheorem{rem}[thm]{Remark}

\author{André Martiny}
\title{On Octahedrality and M{\"u}ntz spaces}

\begin{document}
	\maketitle
	\begin{abstract}
		 We show that every Müntz space can be written as a direct sum of Banach spaces $X$ and $Y$, where $Y$ is almost isometric to a subspace of $c$ and $X$ is finite dimensional. We apply this to show that no Müntz space is locally octahedral or almost square. 
	\end{abstract}

\section{Introduction}

Denote the closed unit ball, the unit sphere, and the dual space of a Banach space $X$  by $B_X$, $S_X$, and $X^*$ respectively. 
Let $\Lambda=(\lambda_i)_{i=0}^\infty$, with $\lambda_0=0$,  be a strictly increasing sequence of non-negative real numbers and let \(\Pi(\Lambda):={\spann}(t^{\lambda_i})_{i=0}^\infty\subseteq C[0,1],	 \) where \(C[0,1]\) is the space of real valued continuous functions on $[0,1]$ endowed with the canonical sup-norm $\|\cdot\|_\infty$. We will call $\Lambda= (\lambda_i)_{i=0}^\infty$ a \emph{Müntz sequence} and  $M(\Lambda):=\overbar{\Pi(\Lambda)}$ a \emph{M{\"u}ntz space} if $\sum_{i=1}^\infty 1/{\lambda_i} <\infty$. This terminology is justified by Müntz famous theorem from 1914, that $\Pi(\Lambda)$ is  dense in $C[0,1]$ \emph{if and only if} $\lambda_0=0$ and $\sum_{i=1}^\infty 1/{\lambda_i} =\infty$.

\begin{defn}
	Let $X$ be a Banach space. Then $X$ is
\begin{enumerate}[label=(\roman*)]\label{def:Lokaltoktahedralt}
		\item \emph{locally octahedral} (LOH) if for every $x\in S_X$ and $\varepsilon >0$ there exists $y\in S_X$ such that $\|x\pm y\|>2-\varepsilon$ \\
		\item \emph{octahedral} (OH) if for every $x_1,\ldots,x_n \in S_X$ and $\varepsilon>0$ there exists $y\in S_X$ such that $\|x_i \pm y\|>2-\varepsilon$ for all $i\in\{1,\ldots,n\}$.
\end{enumerate}
\end{defn}
In this paper we will show that no Müntz space is OH, answering the question posed in \cite{MR3714458} whether Müntz spaces can be OH. A partial negative answer was given in \cite[Remark~2.9]{MR3714458} for Müntz spaces with Müntz sequences consisting only of integers, by combining the Clarkson-Erdös-Schwartz Theorem (see \cite[Theorem 6.2.3]{MR2190706}) with a result of Wojtaszczyk (see \cite[Theorem 1]{werner2008remark}).
 Our answer is based upon a generalization of a result of Werner \cite{werner2008remark}. We show that $M(\Lambda)$ can be written as a direct sum $X\oplus Y$ where $Y$ is almost isometric to a subspace of $c$ and $X$ is finite dimensional. As usual $c$ denotes the Banach space of real valued convergent sequences with the supremum norm.  

\begin{defn}
	Let $X$ be a Banach space. Then $X$ is 
\begin{enumerate}[label=(\roman*)]\label{def:ASQ}
	\item \emph{locally almost square} (LASQ) if for every $x\in S_X$ there exists a sequence $\{y_n\}$ in $B_X$ such that $\|x\pm y_n\| \rightarrow 1$ and $\|y_n\|\rightarrow 1$. \\ 
	\item  \emph{almost square} (ASQ) if for every $x_1,\ldots,x_k\in S_X$ there exists a sequence $\{ y_n\}$ in $B_X$ such that $\|y_n\|\rightarrow 1$ and $\|x_i\pm y_n\| \rightarrow 1$ for every $i\in \{1,\ldots,k\}$.	 
\end{enumerate}
\end{defn}
Both ASQ and OH  are closely related to the area of diameter two properties, which has received  intensive attention in the recent years (see for example  \cite{MR3496032}, \cite{MR3415738}, \cite{MR3714458}, \cite{MR3098474}, \cite{haller2018}). Trivially ASQ implies LASQ and OH implies LOH. 

The area of diameter two properties concerns slices of the unit ball, i.e. subsets of the unit ball of the form 
\[
S(x^*,\varepsilon):=\{ x\in B_X : \ x^*(x)>1-\varepsilon		\},
\]
 where $x^*\in S_{X^*}$ and $\varepsilon>0$. Müntz spaces and their diameter two properties were studied in \cite{MR3714458}. Haller, Langemets, Lima and Nadel \cite{haller2018} pointed out that the proof of \cite[Theorem~2.5]{MR3714458} actually shows that, in any $M(\Lambda)$ we have that for every finite family $\{S_i 	\}_{i=1}^n$ of slices of $B_{M(\Lambda)}$ and $\varepsilon>0$, there exist $x_i\in S_i$ and $y \in B_{M(\Lambda)}$, independent of $i$, such that $x_i\pm y \in S_i$ for every $i\in \{1,\ldots,n\}$ and $\|y\|>1-\varepsilon$. This property is formally known as \emph{the symmetric strong diameter two property }(SSD2P).

It is known that if a Banach space is ASQ then it also has the SSD2P. In fact, ASQ is strictly stronger than SSD2P (see \cite[Theorem 2.1d and Example 2.2]{haller2018}).
A natural question is therefore if a Müntz space can be ASQ. The techniques developed in this article will be used to show that this is never the case.

Note that we can exclude the constants and consider the subspace $M_0(\Lambda):= \cspann\{t^{\lambda_n}\}_{n=1}^\infty$ of $M(\Lambda)$ and the following results still holds true.

The notation $\|f\|_{[0,a]} :=\sup_{x\in [0,a]} |f(x)|$ will be used throughout the paper.
We use standard Banach space terminology and notation (e.g. \cite{MR3526021}).

%%%% 			 RESULTS  				%%%%

\section{Results}

The main results of this article relies on the following theorem (which can be found with proof in \cite{MR1415318}): 
\begin{thm}\label{thm:bound-bern}[Bounded Bernstein's inequality]
	Assume that $1\leq \lambda_1<\lambda_2<\lambda_3 \cdots $ and $ \sum_{i=1}^\infty 1/\lambda_i <\infty$. Then for every  $\varepsilon >0$ there is a constant $c_\varepsilon$ such that 
	\[
	\|p'\|_{[0,1-\varepsilon]} \leq c_\varepsilon \|p\|_{[0,1]},
	\]	
	for all $p\in \Pi(\Lambda)$.
\end{thm}

We will also need a lemma which certainly is well known, but for completeness we include a proof.

%%			NØDVENDIG LEMMA			%%

\begin{lem}\label{lem:remove-fincodim}
	Let $\displaystyle Z=\cspann(z_i)_{i\in \enn}$ and let $N\in \enn$. If $Y=\cspann(z_i)_{i>N}$ then $Z/Y=\spann(\pi(z_i))_{i\leq N}$, where $\pi :Z\ra Z/Y$ is the quotient map. Consequently $Z/Y$ has finite dimension and  $Z=X\oplus Y$ where $X=\spann(x_i)_{i\leq N}$. 
\end{lem}

\begin{proof}
	Let $w\in Z/Y$ and  choose $z\in Z$ such that $w=\pi(z)$. Let $\varepsilon>0$ and find $x=\sum_{i=1}^k a_i z_i\in \spann(z_i)_{i\in \enn}$ such that $\|x-z\| <\varepsilon$. 
	Now 
	\[\|w-\pi(x)\|=\|\pi(z-x)\|<\varepsilon,
	\] 
	 and because $\pi(x)\in \spann(\pi(z_i))_{i\leq N}$, we conclude that \[w=\cspann(\pi(z_i))_{i\leq N}=\spann(\pi(z_i))_{i\leq N}.\]
\end{proof}

\begin{rem}
	For every $N \in \enn$ we have that
	$\cspann(t^{\lambda_i})_{i \ge N}$ is a finite codimensional subspace of
	$M(\Lambda).$
\end{rem}

%%%			ALMOST ISOMETRY	AND CONSEQUENCES		%%%

Now let us show how the Bounded Bernstein inequality can be used to map Müntz spaces $M(\Lambda)$ with $\lambda_1\geq 1$ almost isometrically into subspaces of $c$.

\begin{prop}\label{Prop:almostiso}
	Let $\Lambda$ be a Müntz sequence with $\lambda_1\geq 1$. Then the associated M{\"u}ntz space $M(\Lambda)$ is almost isometric to a subset of $c$.
	That is, for every $\varepsilon>0$ there is an operator $J_{\varepsilon} : M(\Lambda) \rightarrow c$ such that
	\[
	(1-\varepsilon)\|f\|_\infty \leq \|J_\varepsilon f \| \leq \| f \|_\infty
	\]
\end{prop}

\begin{proof}
Let $\varepsilon>0$ and choose a sequence $0=a_0\!<\!a_1\!<\!\cdots\!<\!a_i\!<\!\cdots\!<\!1$ converging to 1. From Theorem \ref{thm:bound-bern} we know that for each $a_i \in (0,1)$ there exists $K_i >0$, depending on $a_i$ such that
\[
\|p'\|_{[0,a_i]} \leq K_i \| p\|_\infty \text{    for all   } p\in \Pi(\Lambda)
\]
Pick points $0=s_0\! <\!s_1\! <\!\cdots\!<\!s_{n_1}\!=\!a_1\! <\!s_{{n_1}+1}\!<\!\cdots\!<\!s_{n_2}\!=\!a_2\!<\!\cdots$,
 in such a way that 
\[
s_{j+1}-s_j \leq \frac{\varepsilon}{K_{i+1}} \text{ for $n_i \leq j < n_{i+1}$}.
\]
Define the operator $J_\varepsilon$ by $ J_\varepsilon (f) = (f(s_n))_n $, thus $\| J_\varepsilon\| \leq 1$ and continuity of $f\in M(\Lambda)$ gives $\lim_{n\rightarrow \infty}f(s_n) \rightarrow f(1)$. \\ 
For any $f\in  M(\Lambda)$ let $(f_k)$ be a sequence in $\Pi(\Lambda)$ converging uniformly to $f$. Let $\delta >0$ and find $N\in \enn$ such that $\|f-f_N\| <  \delta$. Then, for any $s\in [0,1)$, we have $a_i\leq s < a_{i+1}$  for some $i\in \enn$. Let $s_m\in [a_i, a_{i+1}] $ be the member of $(s_n)$ closest to $s$. Then
\begin{align*}
	|f(s)|  \leq |f_N(s)| +\delta & \leq |f_N(s)-f_N(s_m)|+|f_N(s_m)|+\delta \\
	& \leq \sup_{a_i \leq t \leq a_{i+1}} |f_N'(t)||s-s_m|+\|J_\varepsilon f_N\|+\delta \\
	& \leq \|f_N\|_\infty K_{i+1}  \frac{\varepsilon}{K_{i+1}} +\|J_\varepsilon f_N\| +\delta\\
	& \leq \|f_N\|_\infty \varepsilon+ \|J_\varepsilon f_N \|+\delta \\
	& \leq (\|f\|_\infty+\delta)\varepsilon+ (\|J_\varepsilon f\|+\delta)+\delta \\ 
\end{align*}
and therefore  
\[
(1 -\varepsilon)\|f\|_\infty- \delta(\varepsilon+2)\leq \|J_\varepsilon f\|.
\]
Since $\delta$ was arbitrary we conclude that 
\[
(1-\varepsilon)\|f\|_\infty \leq \|J_\varepsilon f \| \leq \| f \|_\infty.
\]

\end{proof}	
We now easily obtain the following.
\begin{thm}\label{thm:directsum}
	Every Müntz space $M(\Lambda)$ can be written as $X\oplus Y$ where $X$ is finite dimensional and $Y$ is almost isometric to a subspace of $c$.
\end{thm}

\begin{proof}
Combine Lemma~\ref{lem:remove-fincodim} and Proposition \ref{Prop:almostiso}.
\end{proof}

%%%			REMARK POLYHEDRALITET 			%%%

\begin{rem}
From \cite[Theorem~10.4.4]{MR2190706} it is known that no Müntz space of dimension greater than 2 is polyhedral. However, from Theorem \ref{thm:directsum} we see that Müntz spaces are isomorphic to a subspace of $c_0$. Since $c_0$ is polyhedral (\cite[Theorem~4.7]{MR0139073}), it follows that any Müntz space can be renormed to be polyhedral.
\end{rem}

%%%			NOT LOH			%%%

By observing that Theorem \ref{thm:directsum} implies that $M(\Lambda)^*$ is separable, we now easily answer the question posed in \cite{MR3714458}. In fact we show more. 

\begin{thm}\label{thm:non-oh}
	No Müntz space $M(\Lambda)$ is LOH.
\end{thm}
\begin{proof}
	Since $M(\Lambda)^*$ is separable, we can combine \cite[Theorem~4.1.3]{MR704815}	with  \cite[Theorem~4.2.13]{MR704815} to see that there exist slices  $S(x,\varepsilon)$ of the unit ball of $M(\Lambda)^*$ of arbitrarily small diameter, where $x$ can be taken from $M(\Lambda)$. By \cite[Theorem 3.1]{MR3346197} this is equivalent to $M(\Lambda)$ failing to be LOH.

\end{proof}

%%%			NOT ASQ			%%%

We finish this article by showing that every Müntz space fails to be ASQ. First we show this for some Müntz spaces.

\begin{prop}\label{prop:Ikkelasq}
	No Müntz space	$M(\Lambda)$ with $\lambda_1\geq 1$ is LASQ.
\end{prop}

\begin{proof}
	Let $\Lambda$ be a Müntz sequence with $\lambda_1\geq 1$ and $M(\Lambda)$ be the associated M{\"u}ntz space. Choose some $x\in(0,1)$. By Theorem \ref{thm:bound-bern} there is a $c\in \enn$ such that $\|f'\|_{[0,x]} \leq c$ for all
	$f\in B_{\Pi(\Lambda)}$. Let $a = \min(\frac{1}{2c}, x)$ and observe that
	\begin{equation*}
		\sup_{f\in B_{\Pi(\Lambda)}} \|f\|_{[0,a]}
	 \leq \frac{1}{2},
	 	\leq \sup_{f\in B_{\Pi(\Lambda)}}\|f\|_{[0,1/2c]}
	\end{equation*}
	since
	\begin{equation*}
		|f(t)| = |f(t) - f(0)|
		\le \|f'\|_{[0,a]} \cdot |t-0|
		\le c \cdot \frac{1}{2c} = \frac{1}{2}.
	\end{equation*}
	
	Recall from \cite[Theorem~2.1]{MR3415738} that $M(\Lambda)$ is LASQ if and only if for every $g\in S_{M(\Lambda)}$ and  $\varepsilon>0$ there exists $h\in S_{M(\Lambda)}$ such that $\|g\pm h\|\leq 1+\varepsilon$. We claim that no such $h$ exists for $g=t^{\lambda_0}$. Indeed, if 	$0 < \varepsilon < a^{\lambda_0}/2$ and $f\in S_{\Pi(\Lambda)}$ is such that
	$\|t^{\lambda_0} \pm f\| \leq 1 + \varepsilon$. Then $|f(t)| < 1 - \varepsilon$ for $t \geq a$ as $t^{\lambda_0} > 2\varepsilon$ for
	$t \geq a$. Thus, $f$ must attain
	its norm on the interval $[0,a]$, contradicting our observation. As $\Pi(\Lambda)$ is dense in
	$M(\Lambda)$, we conclude that $M(\Lambda)$ is not LASQ.

\end{proof}

\begin{prop}
	No M{\"u}ntz space $M(\Lambda)$ is ASQ.
\end{prop}

\begin{proof}
	Combining Lemma \ref{lem:remove-fincodim} with Theorem \ref{prop:Ikkelasq} shows that every Müntz space have a subspace of finite codimension which is not ASQ. By  \cite[Theorem~3.6]{MR3496032} no Müntz space can be ASQ. 
\end{proof}

\bibliographystyle{amsplain}
\providecommand{\bysame}{\leavevmode\hbox to3em{\hrulefill}\thinspace}
\providecommand{\MR}{\relax\ifhmode\unskip\space\fi MR }
% \MRhref is called by the amsart/book/proc definition of \MR.
\providecommand{\MRhref}[2]{%
	\href{http://www.ams.org/mathscinet-getitem?mr=#1}{#2}
}
\providecommand{\href}[2]{#2}

\end{document}